\documentclass[12pt]{amsart}
\usepackage[utf8]{inputenc}
\usepackage[T1]{fontenc}
\usepackage{indentfirst}
\usepackage[active]{srcltx}
\usepackage{amsfonts}
\usepackage{graphicx}
\usepackage{amsmath}
\usepackage{amssymb}
\usepackage{latexsym}
\usepackage{amscd}
\usepackage[all]{xy}
\usepackage{color}
\usepackage{tikz}
\usepackage{breqn}
\usepackage{subfig}
\usetikzlibrary{shapes.geometric}
\usepackage{algorithm}
\usepackage{algorithmic}
\usepackage{hyperref}
\usepackage{cancel}

\newtheorem{theorem}{Theorem}[section]
\newtheorem{lemma}[theorem]{Lemma}
\newtheorem{proposition}[theorem]{Proposition}

\newtheorem{defin}[theorem]{Definition}

\newtheorem{defins}[theorem]{Definitions}

\newtheorem{exs}[theorem]{Examples}

\newtheorem{ex}[theorem]{Example}

\newtheorem{rem}[theorem]{Remark}

\newtheorem{rems}[theorem]{Remarks}

\newtheorem{corollary}[theorem]{Corollary}

\def\e{\epsilon}

\newcommand{\F}{\mathbb{F}}

\def\fq{\mathbb{F}_{q}}

\def\fqt{\mathbb{F}_{q^3}}

\def\e{\epsilon}

\begin{document}
	
\title[On planar functions over $\fqt$]{On planar functions over $\fqt$}
\author[J.P.Guardieiro]{João Paulo Guardieiro}
\address{Centro de Ciências Exatas e Tecnologia - Universidade Federal do Maranhão, Av. dos Portugueses, 1966 - Bacanga, São Luís - MA, Brazil.}
\email{joao.guardieiro@ufma.br}
\address{Instituto de Matemática - Universidade Federal do Rio de Janeiro, Av. Athos da Silveira Ramos, 149 - Cidade Universitária, Rio de Janeiro - RJ, Brazil.}
\author[A. Marques]{Adler Marques}
\email{adler@im.ufrj.br}
\author[L. Quoos]{Luciane Quoos}
\email{luciane@im.ufrj.br}
\author[G. Tizziotti]{Guilherme Tizziotti}
\address{Instituto de Matemática e Estatística - Universidade Federal de Uberlândia, Av. João Naves de Ávila, 2121 - Santa Mônica, Uberlândia - MG, Brazil.}
\email{guilhermect@ufu.br}

\subjclass[2020]{11T06, 11C20, 12E10}
\keywords{Planar functions; finite fields; $q$-polynomials}
\begin{abstract} 
	Let $\mathbb{F}_q$ denote the finite field of order $q$. For $q$ odd, we investigate the planarity over $\mathbb{F}_{q^3}$ of the family
	$$
	f_{E,A,B,C,D}(X) := EX^2+ AX^{q+1}+ BX^{q^2+1}+CX^{2q} +DX^{2q^2}\in \mathbb{F}_{q}[X].
	$$
	Using results from the theory of q-polynomials, we establish conditions under which these polynomials are planar functions. In particular, we provide characterizations for the planarity property and present new families of planar trinomials, quadrinomials, and pentanomials.
\end{abstract}
\maketitle
	\section*{Introduction}
	
	Let $\fq$ be the finite field with $q$ elements, and $\mathbb{F}_{q}^{\ast} $ be the multiplicative group of non-zero elements of $\fq$. Since $\fq$ is finite, every function $f$ from $\mathbb{F}_q$ to itself corresponds to a unique polynomial $f(X)$ in $\mathbb{F}_q[X]$ of degree at most $q-1$. A polynomial $f(X)$ in $\fq[X]$ is called a {\it permutation polynomial} of $\fq$ if the associated polynomial function $f: c \mapsto f(c)$ from $\fq$ to $\fq$ is a bijection. Permutation polynomials over finite fields constitute a classical area of research, principally motivated by their diverse applications in cryptography and coding theory.

	 In \cite{DO}, Dembowski and Ostrom introduced the notion of planar functions in the following way. Let $q$ be odd,  a function $f: \mathbb{F}_{q^n} \rightarrow \mathbb{F}_{q^n}$ is called {\it a planar function} if for each $\epsilon \in \mathbb{F}_{q^n}^{\ast}$, the polynomial defined by $$F_\epsilon(X)=f(X + \epsilon) - f(X)$$ is a permutation polynomial of $\mathbb{F}_{q^n}$. That is, the equations $f(X + \epsilon) - f(X) = \beta$ have exactly one solution for all $\epsilon , \beta \in \mathbb{F}_{q^n}$, with $\epsilon \neq 0$. Note that, if $q$ is even, then $F_{\epsilon}(0) = F_{\epsilon}(\epsilon)$ for every $\epsilon \neq 0$. Therefore, functions satisfying the previous condition exist only for $q$ odd. In the same paper, Dembowski and Ostrom also observed that such functions have close connections with projective planes. Zhou \cite{zhou} defined a natural analogue of planar functions for even characteristic in the following way: if $q$ is even, a polynomial $f \in \mathbb{F}_{q}[X]$ is planar on $\mathbb{F}_{q}$ if $$ f(X + \epsilon) + f(X) + \epsilon X$$ is a permutation polynomial on $\mathbb{F}_{q}$ for each $\epsilon \in \mathbb{F}_{q}^{\ast}$. As shown by Zhou \cite{zhou} and Schmidt and Zhou \cite{SZ}, such planar functions have similar properties and applications as their odd characteristic counterparts. 
	  
	  Planar functions in odd characteristic achieve optimal differential properties, making them useful for constructing cryptographic protocols, particularly for designing S-boxes in block ciphers, see \cite{nyberg} and \cite{cryp}, for error-correcting codes see \cite{DXYY} and \cite{YCD}, and applications in signal sets \cite{DY}. In \cite{zieve}, Zieve studied finite projective planes involving planar functions. In particular, he determined all planar functions on $\mathbb{F}_q$ of the form $a \mapsto a^t$, under the assumption that $q \geq (t-1)^4$. As a consequence, the results presented by Zieve solved two conjectures of Hernando, McGuire and Monteserrat \cite{HMM} and also yield a new proof of a conjecture of Segre and Bartocci \cite{SB} about monomial hyperovals in finite Desarguesian projective planes. The existence and non-existence of planar mappings have been extensively studied in the literature, using various methods, see e.g. \cite{BN2015, BH2011, CSZ, CM1997, KO2012}. Constructing planar polynomials, or characterizing the conditions under which they arise, is a problem in finite field theory that has drawn the attention of many researchers in recent years. Bartoli and Timpanella \cite{bartoli2} investigated planar polynomials of the type
	   $$f_{A,B}(X) = AX^{2^{2m} +1} + B X^{2^m + 1}, \text{ where } A,B \in \mathbb{F}_{2^{3m}}^{\ast}.$$
	    In \cite{bartoli}, Bartoli and Bonini construct planar polynomials over $\mathbb{F}_{q^3}$ of the type 
	    $$f(X) = X(X^{q^2} + AX^q + X), \text{ where } A,B \in \mathbb{F}_q.$$
	     In \cite{CM}, Chen and Mesnager, presented a class of planar functions of the form $$Tr(aX^{q+1}) + \ell{X^2}$$ on $\mathbb{F}_{q^n}$, in particular, they construct several typical kinds of planar functions on $\mathbb{F}_{q^2}$, characterize them on $\mathbb{F}_{q^3}$, and proved that such planar functions do not exist given certain conditions when the degree of the extension of $\mathbb{F}_q$ is higher. Chan and Xiong \cite{CX} characterized planar quadrinomials of the form $$f_{\underline{c}}(X)=c_0X^{qQ+q}+c_1X^{qQ+1}+c_2X^{Q+q}+c_3X^{Q+1}$$ over $\mathbb{F}_{q^2}$, where $Q$ is a power of the characteristic of the field, for any $\underline{c}=(c_0,c_1,c_2,c_3)$ in $ \mathbb{F}_{q^2}^4$ in terms of linear equivalence. For a survey with results on planar functions see e.g. \cite{pott}. 
	  
	  For $q$ odd, in this work  we investigate the planarity over $\mathbb{F}_{q^3}$ of the following family of pentanomials 
	  $$f_{E,A,B,C,D}(X) := EX^2+AX^{q+1}+ BX^{q^2+1}+CX^{2q} +DX^{2q^2} \in \mathbb{F}_{q}[X].$$
	  In particular, we obtain the following new  families of planar functions over $\mathbb{F}_{q^3}$ for $q$ odd:
	  
	  \begin{enumerate}  	
	  	\item  $EX^2 + CX^{2q} +DX^{2q^2}$, with $C, D, E \in \mathbb{F}_q$ satisfying $C^3 + D^3 + E^3 - 3CDE=1/2$. (see Theorem \ref{samefactor}).
	  	\item $-X^2 + 2X^{q^2+1} + X^{2q} - X^{2q^2}$ (see Theorem \ref{teo1}).
		\item $EX^2 + 2(E-D)X^{q+1} + 2DX^{q^2+1} + (E-D)X^{2q} + DX^{2q^2}$,  with $ D, E \in \mathbb{F}_q$ satisfying $E(3D^2 - 3DE + E^2) \neq 0$ (see Theorem \ref{2parametros}).
			\item $X^2 - X^{q+1} - X^{q^2 + 1} + X^{2q} + X^{2q^2}$ (see Theorem \ref{teo: two pent}).
		\item $X^2 + X^{q+1} + X^{q^2 + 1} + X^{2q} + X^{2q^2}$ (see Theorem  \ref{teo: two pent}).	
	  \end{enumerate}
	  We organize this work in the following way. Section 1 is devoted to presenting some notations and results related to the concept of the $q$-polynomial over the finite field $\mathbb{F}_q$. In Section 2, we present the main results of the paper, which establish conditions for the planarity of pentanomials as described above. Finally, Section 3 provides illustrative examples of planar polynomials.

	\section{Preliminaries}
	
	In this section, we  present some results involving the concept of $q$-polynomials, which will be used throughout this paper.
	
	A polynomial of the form $\displaystyle L(X)=\sum_{i=0}^{n-1} c_i  X^{q^i} \in \mathbb{F}_{q^n}[X]$ is called a $q$-polynomial (or a $q$-linearized polynomial) over $\mathbb{F}_{q^n}$. A $q$-polynomial $L(X)$ induces a linear transformation $L: \mathbb{F}_{q^n} \rightarrow \mathbb{F}_{q^n}$, defined by $a \mapsto L(a)$, where $\mathbb{F}_{q^n}$ is considered as an $\mathbb{F}_{q}$-vector space. The following theorem provides a simple, yet powerful, criterion for a 
$q$-polynomial to induce a permutation of $\mathbb{F}_{q^n}$.
	
	\begin{theorem}\cite[Th. 7.9]{finite}
		Let $\mathbb{F}_{q^n}$ be an $\mathbb{F}_q$-vector space. Then the $q$-polynomial
		$$
		\displaystyle L(X) = \sum_{i=0}^{n-1} c_i X^{q^i} \in \mathbb{F}_{q^n}[X] 
		$$
		is a permutation polynomial of $\mathbb{F}_{q^n}$ if and only if $L(X)$ only has the $0$ root in $\mathbb{F}_{q^n}$.
	\end{theorem}
	
This criterion based on the roots of the $q$-polynomial $
	\displaystyle L(X)
	$ is not always the most convenient tool for explicit computations. An equivalent condition can be given in terms of the Dickson matrix associated with the polynomial. This matrix, known as the $q$-circulant matrix, represents the action of the polynomial on a basis of $\mathbb{F}_{q^n}$ over $\mathbb{F}_q$. The following lemma translates the permutation property of   $
	\displaystyle L(X)
	$ into a nonvanishing determinant condition on the Dickson matrix, a result that will be particularly useful in the next section.

	\begin{lemma}\cite[p. 362]{finite}\label{LemmaLinearized}
		Let $q$ be a prime power, $\mathbb{F}_{q^n}$ an extension of $\mathbb{F}_{q}$, and consider
		$$L(x)=\sum_{i=0}^{n-1} c_i x^{q^i}\in \mathbb{F}_{q^n}[x].$$
		Let $D_L$ be the associated Dickson matrix, defined as
		$$D_L:=
		\left(
		\begin{array}{cccccc}
			c_0&c_1&c_2&\cdots&c_{n-1}\\
			c_{n-1}^q&c_0^q&c_1^q&\cdots&c_{n-2}^q\\
			c_{n-2}^{q^2}&c_{n-1}^{q^2}&c_0^{q^2}&\cdots&c_{n-3}^{q^2}\\
			\vdots&\vdots&\vdots&&\vdots\\
			c_{1}^{q^{n-1}}&c_{2}^{q^{n-1}}&c_3^{q^{n-1}}&\cdots&c_{0}^{q^{n-1}}\\
		\end{array}
		\right).
		$$ Then $L(x)$ is a permutation polynomial of $\mathbb{F}_{q^n}$ if and only if 
		$\det(D_L) \neq 0$.
	\end{lemma}	
	
	Throughout this work, to simplify the notation, for a polynomial $f(X) \in \mathbb{F}_q[X]$ we define for each $\epsilon \in \mathbb{F}_q$ the polynomial 
	\[
	F_{\epsilon}(X) := f(X+\epsilon) - f(X) - f(\epsilon).
	\] 
	In the case the polynomial $F_{\epsilon}(X)  $ is a $q$-polynomial, the determinant of the Dickson matrix associated with $F_{\epsilon}(X)$ will be denoted by $\Delta_{F_{\epsilon}}$.
	
	\section{Planar functions over \texorpdfstring{$\fqt$}{Fq3}}

		In this section, we investigate the planarity of the following family  over $\mathbb{F}_{q^3}[X]$. Let $q$ be a power of an odd prime, and  
consider
	\begin{equation}\label{eq:generalfamily1}
		f_{E,A,B,C,D}(X) =  EX^2+AX^{q+1}+ BX^{q^2+1}+CX^{2q} +DX^{2q^2} \in \mathbb{F}_{q}[X].
	\end{equation}

A characterization of when such a pentanomial is planar over $\mathbb{F}_{q^3}$ is established in the next result. 
		\begin{proposition}\label{Prop:Condition}
		Let  $A,B,C,D,E\in \mathbb{F}_{q}$, not all zero. The polynomial $f_{E,A,B,C,D}(X)  \in \mathbb{F}_{q^3}[X]$ given in \eqref{eq:generalfamily1} is planar if and only if 
		\begin{align*}
			&(-2A^2D + 2ABE - 2B^2C)(\epsilon^3 + \epsilon^{3q} + \epsilon^{3q^2})+ \\
			&(2A^2C - 2ABD - 4ACD + 4AE^2 + 2B^2E - 4BCE + 4BD^2)(\epsilon^{2+q} + \epsilon^{1+2q^2} + \epsilon^{2q+q^2})+ \\
			&(2A^2E - 2ABC + 4AC^2 - 4ADE + 2B^2D - 4BCD + 4BE^2)(\epsilon^{2+q^2} + \epsilon^{1+2q} + \epsilon^{q+2q^2})+ \\
			&(2A^3 + 2B^3 + 8C^3 - 24CDE + 8D^3 + 8E^3)\epsilon^{1+q+q^2} \neq 0
		\end{align*} 
		for any  $\epsilon\in \mathbb{F}_{q^3}^*$.
	\end{proposition}
	\begin{proof}
		By the definition of planar function, we have that the polynomial $f_{E,A,B,C,D}$ is planar if and only if, for all
		$ \epsilon \in \mathbb{F}_{q^3}^*$, the associated polynomial $$F_{\epsilon}(X):=f_{E,A,B,C,D}(X+\epsilon)-f_{E,A,B,C,D}(X)-f_{E,A,B,C,D}(\epsilon)$$ is a permutation polynomial over $\mathbb{F}_{q^3}$.  We have 
		\begin{align*}
			F_{\epsilon}(X)&= E(X+\e)^2+A(X+\e)^{q+1}+ B(X+\e)^{q^2+1}+C(X+\e)^{2q} +D(X+\e)^{2q^2}\\
			&-(EX^2+ AX^{q+1}+ BX^{q^2+1}+CX^{2q} +DX^{2q^2})\\
			&-(E\e^2 + A\e^{q+1}+ B\e^{q^2+1}+C\e^{2q} +D\e^{2q^2})\\
			&=(2E\e + A \epsilon^q + B \epsilon^{q^2})X + (A \epsilon + 2C \epsilon^q)X^q + (B \epsilon + 2D \epsilon^{q^2})X^{q^2}.
		\end{align*}

		Since $F_{\epsilon}(X)$  is a $q$-polynomial, by Lemma \ref{LemmaLinearized}, $F_{\epsilon}(X)$ is a permutation polynomial of $\fqt$ for any $\epsilon \in \mathbb{F}_{q^3}^*$  if and only if its associated Dickson matrix is  nonsingular for any $\epsilon \in \mathbb{F}_{q^3}^*$. Now, the determinant $\Delta_{F_{\epsilon}}$ of the Dickson matrix of $F_{\epsilon}(X)$ is given by 
		\[
		\Delta_{F_{\epsilon}} = 
		\det\begin{pmatrix}
			2E\e + A \epsilon^{q} + B \epsilon^{q^2} &  A \epsilon + 2 C \epsilon^{q}      	&B \epsilon + 2 D \epsilon^{q^2} \\
			B \epsilon^{q} + 2 D \epsilon 		   & 2E\e^q + A \epsilon^{q^2} + B \epsilon  &    A \epsilon^{q} + 2 C \epsilon^{q^2} \\
			A \epsilon^{q^2} + 2 C \epsilon  &   B \epsilon^{q^2} + 2 D \epsilon^{q} 	& 2E\e^{q^2} + A \epsilon + B \epsilon^{q}\\
		\end{pmatrix}, 
		\]
		which simplifies exactly to the expression in the statement. So, the result follows from Lemma \ref{LemmaLinearized}. 
	\end{proof}
	
	This characterization reduces the planarity condition to  a polynomial not having zeros in  $\mathbb{F}_{q^3}^\ast$. Although explicit, this condition is somewhat complex and not easily verified directly. In the following, we present some results to enable explicit constructions of planar pentanomials.

\begin{lemma} \label{Lemma_L}
	Let $q$ be an odd prime power. For $a, b, c \in \mathbb{F}_q$, the polynomial 
$$L(X) = aX^{q^2} + bX^q+cX$$ has no roots in $\mathbb{F}_{q^3}^*$ if and only if $a^3+b^3+c^3-3abc \neq 0$.
\end{lemma}
\begin{proof}
Since  $L(X)$ is a $q$-polynomial over $\mathbb{F}_{q^3}$, by Lemma \ref{LemmaLinearized}, it is a permutation polynomial if and only if the determinant of its respective associated Dickson matrix is non-zero.
 The associated Dickson matrix of $L(X)$ is
$$
A_{L} = \begin{pmatrix}
	a& b & c\\
	c & a & b\\
	b & c & a
\end{pmatrix},
$$
which has determinant $\det(A_{L}) = a^3+b^3+c^3-3abc$.
Since $L(0)=0$ and $L$ permutes $\mathbb{F}_{q^3}$, we conclude that it has no root in $\mathbb{F}_{q^3}^*$.
\end{proof}

In the next  results 
we examine two families of $3 \times 3$ matrices, where the entries are functions of an element $\epsilon \in \mathbb{F}_{q^3}^*$. For each of these families of matrices, we are going to show that the determinant is never zero, regardless of the value chosen for $\epsilon \in \mathbb{F}_{q^3}^*$. These technical results will play a key role in verifying the nonvanishing condition that appears in the planarity characterization given in Proposition \ref{Prop:Condition}.

\begin{proposition} \label{lemma Ag}
	Let $q$ be  odd. For $\e \in \mathbb{F}_{q^3}^{\ast}$ consider the matrix
	$$
	A_{\e}=\left( \begin{array}{ccc}
		2 \epsilon + \epsilon^q + \epsilon^{q^2} & ( \epsilon + \epsilon^{q^2})^q & (\epsilon + \epsilon^q)^{q^2}\\
		\epsilon + \epsilon^q & (2 \epsilon + \epsilon^q + \epsilon^{q^2})^q & ( \epsilon + \epsilon^{q^2} )^{q^2}\\
		\epsilon + \epsilon^{q^2} & (\epsilon + \epsilon^q)^q & (2 \epsilon + \epsilon^q + \epsilon^{q^2})^{q^2}
	\end{array}\right).
	$$
	Then, $\det(A_{\e}) \neq 0$ for all $\e \in \mathbb{F}_{q^3}^{\ast}$. 
\end{proposition}
\begin{proof}
	The determinant of the matrix $A_\e$ is given by
	\begin{align*}
		\det(A_{\e}) &= (2 \e + \e^q + \e^{q^2})^{q^2 + q + 1} + (\e + \e^q)^{q^2 + q +1} + (\e + \e^{q^2})^{q^2 + q + 1}\\
		& - (2 \e + \e^q + \e^{q^2})^q (\e + \e^q)^{q^2} (\e + \e^{q^2})\\
		& - (2 \e + \e^q + \e^{q^2}) (\e + \e^q)^{q} (\e + \e^{q^2})^{q^2}\\
		& - (2 \e + \e^q + \e^{q^2})^{q^2} (\e + \e^q) (\e + \e^{q^2})^q .
	\end{align*}
			
		To obtain an expression for $\det(A_{\e})$, we do the calculations separately, using the facts that $\epsilon \in \mathbb{F}_{q^3}^{\ast}$ and that $q$ is odd. 
	Define 
	\[
	P:=(2 \e + \e^q + \e^{q^2})^{q^2 + q + 1} + (\e + \e^q)^{q^2 + q +1} + (\e + \e^{q^2})^{q^2 + q + 1},
	\]
	and 
	\begin{align*}
		N&=(2 \e + \e^q + \e^{q^2})^q (\e + \e^q)^{q^2} (\e + \e^{q^2})+(2 \e + \e^q + \e^{q^2}) (\e + \e^q)^{q} (\e + \e^{q^2})^{q^2}\\
		&+(2 \e + \e^q + \e^{q^2})^{q^2} (\e + \e^q) (\e + \e^{q^2})^q.
	\end{align*}
	Then $\det(A_{\e})=P-N$.
	
	At first we compute $P$,
	\begin{align*}
		P &= (2 \e^{q^2} + \e + \e^{q}) (2 \e^q + \e^{q^2} + \e) (2 \e + \e^q + \e^{q^2}) + (\e^{q^2} + \e)(\e^q + \e^{q^2})(\e + \e^{q})\\
		& + (\e^{q^2} + \e^q) (\e^q + \e)(\e + \e^{q^2})\\
		&= 2 \e^{3 q^2} + 9 \e^{2q^2 + q} + 9 \e^{2q^2 + 1} + 9 \e^{q^2 + 2q} + 20 \e^{q^2 + q + 1} + 9 \e^{q^2 + 2} + 2 \e^{3q} + 9 \e^{2q+1}\\
		& + 9 \e^{q+2} + 2\e^3.
	\end{align*} 
	
	Now, we compute $N$,
	\begin{align*}
		N &= (2 \e + \e^q + \e^{q^2})^q (\e + \e^q)^{q^2} (\e + \e^{q^2})\\
		&+(2 \e + \e^q + \e^{q^2}) (\e + \e^q)^{q} (\e + \e^{q^2})^{q^2}\\
		&+(2 \e + \e^q + \e^{q^2})^{q^2} (\e + \e^q) (\e + \e^{q^2})^q\\
		&=\e^{3q^2} + 2 \e^{2q^2 + q} + 3 \e^{2q^2 + 1} + 4 \e^{q^2 + q + 1} + 3 \e^{q^2 + 2} + 2 \e^{q+2} + \e^3\\
		&+\e^{3q^2} + 3\e^{2q^2 + q} + 2 \e^{2q^2 + 1} + 3 \e^{q^2 + 2q} + 4 \e^{q^2 + q + 1} + \e^{3q} + 2 \e^{2q+1}\\
		&+2 \e^{q^2 + 2q} + 4 \e^{q^2 + q + 1} + 2 \e^{q^2 + 2} + \e^{3q} + 3 \e^{2q+1} + 3 \e^{q+2} + \e^3\\
		&= 2(\e^{3q^2} + \e^{3q} + \e^3) + 5(\e^{2q^2+q} + \e^{2q^2+1} + \e^{q^2+2q} + \e^{2q+1} + \e^{q^2+2} + \e^{q+2}) + 12\e^{q^2+q+1}.
	\end{align*}

	Then, we have that 
	$$
	\det(A_{\e}) = 4 \e^{2q^2 + q} + 4 \e^{2q^2 + 1} + 4 \e^{q^2 + 2q} + 8 \e^{q^2 + q + 1} + 4 \e^{q^2 + 2} + 4 \e^{2q+1} + 4 \e^{q+2}.
	$$
	So, we get 
	\[
	\det(A_{\e}) = 4 (\e^{q^2}+\e^q)(\e^{q^2}+\e)(\e^q+\e).
	\]
	By Lemma \ref{Lemma_L}, for $p(X)=X^{q^2}+X^q \in \F_{q^3}[X]$, we have $\e^{q^2}+\e^q \neq 0$ for all $\e \in \F_{q^3}^\ast$, thus we also have  $\e^{q^2}+\e \neq 0$ and $\e^{q}+\e \neq 0$ for all $\e \in \F_{q^3}^\ast$. Hence, $\det(A_\e) \neq 0$ for all $\e \in \mathbb{F}_{q^3}^{\ast}$.
\end{proof}

\begin{proposition} \label{lemma Bh}
	Let $q$ be a power of an odd prime. For $\e \in \mathbb{F}_{q^3}^{\ast}$ consider the matrix
	$$
	B_{\e}:=\left( \begin{array}{ccc}
		2 \epsilon - \epsilon^q - \epsilon^{q^2} & ( 2\epsilon^{q^2} - \epsilon)^q & (2\epsilon^q -\epsilon)^{q^2}\\
		2\epsilon^q -\epsilon & (2 \epsilon - \epsilon^q - \epsilon^{q^2})^q & ( 2\epsilon^{q^2} - \epsilon )^{q^2}\\
		2\epsilon^{q^2} - \epsilon & (2\epsilon^q -\epsilon)^q & (2 \epsilon - \epsilon^q - \epsilon^{q^2})^{q^2}
	\end{array}\right).
	$$
	Then, $\det(B_{\e}) \neq 0$ for all $\e \in \mathbb{F}_{q^3}^{\ast}$. 
\end{proposition}

\begin{proof}
	The determinant of the matrix $
	B_{\e}$  is given by
	$$
	\begin{array}{rl}
		\det(B_{\e}) &= (2 \epsilon - \epsilon^q - \epsilon^{q^2})^{q^2 + q + 1} + (2\epsilon^q -\epsilon)^{q^2 + q +1} + (2\epsilon^{q^2} - \epsilon)^{q^2 + q + 1}\\
		& - (2 \epsilon - \epsilon^q - \epsilon^{q^2})^q (2\epsilon^q -\epsilon)^{q^2} (2\epsilon^{q^2} - \epsilon)\\
		& - (2 \epsilon - \epsilon^q - \epsilon^{q^2}) (2\epsilon^q -\epsilon)^{q} (2\epsilon^{q^2} - \epsilon)^{q^2}\\
		& - (2 \epsilon - \epsilon^q - \epsilon^{q^2})^{q^2} (2\epsilon^q -\epsilon) (2\epsilon^{q^2} - \epsilon)^q .\\
	\end{array}
	$$
	To obtain an expression for $\det(B_{\e})$, we do the calculations separately, using the facts that $\epsilon \in \mathbb{F}_{q^3}^{\ast}$ and that $q$ is odd.
	
	$\bullet$ The first line of the determinant reduces to:
	
	$$\begin{array}{l}
		(2 \epsilon - \epsilon^q - \epsilon^{q^2})^{q^2 + q + 1} + (2\epsilon^q -\epsilon)^{q^2 + q +1} + (2\epsilon^{q^2} - \epsilon)^{q^2 + q + 1}\\
		\mbox{ }= (2 \epsilon^{q^2} - \epsilon - \epsilon^{q}) (2 \epsilon^q - \epsilon^{q^2} - \epsilon) (2 \epsilon - \epsilon^q - \epsilon^{q^2})\\
		\mbox{ } \mbox{ } \mbox{ } \mbox{ } + (2\epsilon -\epsilon^{q^2})(2\epsilon^{q^2} -\epsilon^q)(2\epsilon^q -\epsilon) + (2\epsilon^{q} - \epsilon^{q^2}) (2\epsilon - \epsilon^q)(2\epsilon^{q^2} - \epsilon)\\
		\mbox{ } = 2\e^{3 q^2} -5\e^{2q^2 + q} -5 \e^{2q^2 + 1} - 5\e^{q^2 + 2q} + 26\e^{q^2 + q + 1} - 5 \e^{q^2+2} + 2\e^{3q}\\
		\mbox{ } \mbox{ } \mbox{ } \mbox{ }  - 5\e^{2q+1} - 5 \e^{q+2} +2\e^3.
	\end{array}$$
	
	\bigskip
	
	$\bullet$ Now we simplify, one by one,  the second, third and fourth line of the determinant:\\
	
	$\diamond$ $(2 \epsilon - \epsilon^q - \epsilon^{q^2})^q (2\epsilon^q -\epsilon)^{q^2} (2\epsilon^{q^2} - \epsilon) = 2\e^{3q^2} - 4 \e^{2q^2 + q} - 3 \e^{2q^2 + 1} + 10 \e^{q^2 + q + 1} - 3\e^{q^2 + 2} -4\e^{q + 2} + 2 \e^{3}$.
	
	\medskip
	
	$\diamond$ $(2 \epsilon - \epsilon^q - \epsilon^{q^2}) (2\epsilon^q -\epsilon)^{q} (2\epsilon^{q^2} - \epsilon)^{q^2} = 2\e^{3q^2} -3\e^{2q^2 + q} - 3 \e^{q^2 + 2q} - 4\e^{2q^2 + 1} + 10 \e^{q^2 + q + 1}+2\e^{3q} -4\e^{2q+1}$.
	
	\medskip
	
	$\diamond$ $(2 \epsilon - \epsilon^q - \epsilon^{q^2})^{q^2} (2\epsilon^q -\epsilon) (2\epsilon^{q^2} - \epsilon)^q = -4\e^{q^2 + 2q} + 10 \e^{q^2 + q + 1} - 4 \e^{q^2 + 2} + 2 \e^{3q} - 3\e^{2q+1} -3\e^{q+2} + 2\e^3$.
	
	\medskip
	
	And we get that the sum of the three last lines is given by:
	
	$-(	4\e^{3q^{2}} - 7\e^{2q^{2}+q} - 7\e^{2q^{2}+1} - 7\e^{q^{2}+2q} + 30\e^{q^{2}+q+1} - 7\e^{q^{2}+2} + 4\e^{3q} - 7\e^{2q+1} - 7\e^{q+2} + 4\e^{3}).$
	
	\medskip

	Finally, we conclude that
	$$
	\det(B_{\e}) = -2\e^{3q^2} + 2 \e^{2q^2 + q} + 2 \e^{2q^2 + 1} + 2 \e^{q^2 + 2q} -4\e^{q^2 + q + 1} + 2\e^{q^2 + 2} -2\e^{3q} + 2 \e^{2q+1} + 2 \e^{q+2} -2\e^3.
	$$

	Then, for 
	$$
	\begin{array}{rl}
		p(X):=& -2X^{3q^2} + 2 X^{2q^2 + q} + 2 X^{2q^2 + 1} + 2 X^{q^2 + 2q} -4X^{q^2 + q + 1} + 2X^{q^2 + 2} -2X^{3q}  \\
		& + 2 X^{2q+1} + 2 X^{q+2} -2X^3 \in \mathbb{F}_{q^3}[X].
	\end{array}
	$$
	we have  $\det(B_{\e}) = p(\e)$. So, $\det(B_{\e}) \neq 0$ for every $\e \in \mathbb{F}_{q^3}^{\ast}$ if and only if $p(X)$ has no roots in $\mathbb{F}_{q^3}^{\ast}$.
	
	Note that, $p(X)$ factors exactly as 
	$$p(X) = 2 (X^{q^2}-X^q+X)(X^{q^2}+X^q - X)(-X^{q^2}+X^q + X).$$
	
	By Lemma \ref{Lemma_L}, the polynomial $L(X)=X^{q^2}-X^q+X$  has no roots in $\mathbb{F}_{q^3}^*$. 
	The other two factors of $p(X)$ in the product are simply the $q$-th and $q^2$-th powers of $L(X)$ over $\F_{q^3}$, so they also have no roots in $\mathbb{F}_{q^3}^*$. We conclude that $p(X)$ has no roots in $\mathbb{F}_{q^3}^{\ast}$. 
	
\end{proof}

	According to Proposition \ref{Prop:Condition}, the planarity of $f_{E,A,B,C,D}$ reduces to the condition that the determinant $\Delta_{F_{\epsilon}}$ is nonzero for all $\epsilon \in \mathbb{F}_{q^3}^*$. A direct treatment of this determinant in terms of arbitrary coefficients is, however, algebraically non-trivial. However, motivated by the organized factorizations observed in  Propositions \ref{lemma Ag} and \ref{lemma Bh}, we can seek specific algebraic conditions on the coefficients $A, B, C, D$, and $E$ that forces $\Delta_{F_{\epsilon}}$ to be factored into a product of simpler $q$-polynomials. If we can guarantee that none of these factors have roots in $\mathbb{F}_{q^3}^*$, then the determinant $\Delta_{F_{\epsilon}}$ will remain nonzero, automatically satisfying the condition of planarity. 
	
	The following theorem establishes a general framework for this approach. It provides a specific system of equations over $\mathbb{F}_q$ which, when satisfied, guarantees that the determinant $\Delta_{F_{\epsilon}}$ factors exactly into three linear terms with no non-trivial roots.

		\begin{theorem}\label{prod-polis}
		Let
		$P_i(X) = a_iX^{q^2} + b_iX^q + c_iX, \text{ for } i=1, 2, 3$
		be polynomials in $\mathbb{F}_q[X]$ with no roots in $\mathbb{F}_{q^3}^*$. Suppose that
		\begin{enumerate}
			\item $a_1a_2a_3 = b_1b_2b_3 = c_1c_2c_3,$
			\item $a_1a_2b_3 + a_1b_2a_3 + b_1a_2a_3 = b_1b_2c_3 + b_1c_2b_3 + c_1b_2b_3 = a_1c_2c_3 + c_1c_2a_3 + c_1a_2c_3,$ and 
			\item $a_1a_2c_3 + a_1c_2a_3 + c_1a_2a_3 = b_1c_2c_3 + c_1c_2b_3 + c_1b_2c_3 = a_1b_2b_3 + b_1a_2b_3 + b_1b_2a_3.$
		\end{enumerate}
		Define
		$$\alpha = a_1a_2a_3, \quad \beta = a_1a_2b_3 + a_1b_2a_3 + b_1a_2a_3, \quad \gamma = a_1a_2c_3 + a_1c_2a_3 + c_1a_2a_3$$
		and
		$$\delta = a_1b_2c_3 + a_1c_2b_3 + b_1a_2c_3 + b_1c_2a_3 + c_1a_2b_3 + c_1b_2a_3.$$
		If the system of equations
		\begin{equation} \label{sistema}
		\begin{cases}
				-A^2D + ABE - B^2C = 2\alpha \\
				A^2C - ABD - 2ACD + 2AE^2 + B^2E - 2BCE + 2BD^2 = 2\beta \\
				A^2E - ABC + 2AC^2 - 2ADE + B^2D - 2BCD + 2BE^2 = 2\gamma \\
				A^3 + B^3 + 4C^3 - 12CDE + 4D^3 + 4E^3 = 2\delta
		\end{cases}
		\end{equation}
		has a solution $(E, A, B, C, D)$ in $\mathbb{F}_q^5$, then the polynomial
		$$f_{E,A,B,C,D}(X) = EX^2 + AX^{q+1}+ BX^{q^2+1}+CX^{2q} +DX^{2q^2}$$
		is planar.
	\end{theorem}
	\begin{proof}
	By the system of equations, we can rewrite the determinant $\Delta_{F_{\epsilon}}$ given in Proposition \ref{Prop:Condition} as
		\begin{equation} \label{determinante teorema}
		\Delta_{F_{\epsilon}} = \alpha(\epsilon^3 + \epsilon^{3q} + \epsilon^{3q^2}) + \beta(\epsilon^{2+q} + \epsilon^{1+2q^2} + \epsilon^{2q+q^2}) + \gamma(\epsilon^{2+q^2} + \epsilon^{1+2q} + \epsilon^{q+2q^2}) + \delta \epsilon^{1+q+q^2}.
		\end{equation}
		By the definitions of $\alpha, \beta, \gamma$ and $\delta$, this expression expands and factors exactly as the product of the three $q$-polynomials evaluated at $\epsilon$
		$$
		P_1(\e)P_2(\e)P_3(\e) = (a_1\epsilon^{q^2} + b_1\epsilon^q + c_1\epsilon)(a_2\epsilon^{q^2} + b_2\epsilon^q + c_2\epsilon)(a_3\epsilon^{q^2} + b_3\epsilon^q + c_3\epsilon).
		$$
	Since by hypothesis none of these $q$-polynomials have roots in $\F_{q^3}^*$, it follows that $P_i(\epsilon) \neq 0$ for all $\e \in \mathbb{F}_{q^3}^*$ and $i \in \{1, 2, 3\}$. Therefore, $\Delta_{F_{\epsilon}} \neq 0$ for all $\epsilon \in \mathbb{F}_{q^3}^*$, and by Proposition \ref{Prop:Condition}, it follows that $f_{E,A,B,C,D}(X)$ is planar. 
	\end{proof}
	
	\begin{corollary}\label{cor1poli}
		Let $aX^{q^2} + bX^q + cX \in \mathbb{F}_q[x]$ be a polynomial with no roots in $\mathbb{F}_{q^3}^*$. Define
		$$\alpha = abc, \quad \beta = a^2b + ac^2 + b^2c, \quad \gamma = a^2c + ab^2 + bc^2, \quad \delta = a^3 + b^3 + c^3 + 3abc.$$
		If the system of equations
		\begin{equation*}
			\begin{cases}
				-A^2D + ABE - B^2C = 2\alpha \\
				A^2C - ABD - 2ACD + 2AE^2 + B^2E - 2BCE + 2BD^2 = 2\beta \\
				A^2E - ABC + 2AC^2 - 2ADE + B^2D - 2BCD + 2BE^2 = 2\gamma \\
				A^3 + B^3 + 4C^3 - 12CDE + 4D^3 + 4E^3 = 2\delta
			\end{cases}
		\end{equation*}
		has a solution $(A, B, C, D, E)$ in $\mathbb{F}_q^5$, then the polynomial
		$$f_{E,A,B,C,D}(X) = EX^2 + AX^{q+1}+ BX^{q^2+1}+CX^{2q} +DX^{2q^2}$$
		is planar.
	\end{corollary}
	\begin{proof}
		This is a direct consequence of the previous theorem, by considering the polynomials
		\begin{align*} 
			&P_1(X)=aX^{q^2} + bX^q + cX,\\ 
			&P_2(X)=bX^{q^2} + cX^q + aX, \\
			&P_3(X)=cX^{q^2} + aX^q + bX,
		\end{align*} 
		and noticing that for $\epsilon \in \mathbb{F}_{q^3}^\ast$ we have $P_2(\epsilon)=P_1(\epsilon)^q$ and $P_3(\epsilon)=P_1(\epsilon)^{q^2}$.
	\end{proof}

	\section{Explicit Families of Planar Polynomials}
	
	In this section, we apply the results given in the previous section to construct explicit families of planar polynomials. By choosing suitable $q$-polynomials that have no roots in $\mathbb{F}_{q^3}^*$, we obtain the parameters $\alpha, \beta, \gamma, \delta$ introduced in Theorem \ref{prod-polis}. Solving the resulting system of equations yields several distinct families of planar functions, which are presented below.
	
	In the next theorem we provide a family of  planar trinomials over $\mathbb{F}_{q^3}$. This theorem comes from the most simple choice of $q$-polynomials having no roots in $\mathbb{F}_{q^3}^*$, namely $X^{q^2}$, $X^q$ and $X$.

	\begin{theorem}\label{samefactor}
		Let $q$ be odd and consider the polynomial
		$$f(X) = EX^2 + CX^{2q} +DX^{2q^2}\in \mathbb{F}_q[X].
		$$
		If $C^3 + D^3 + E^3 - 3CDE = 2^{-1}$,  then $f(X)$ is planar over $\mathbb{F}_{q^3}$.  
	\end{theorem}
	
	\begin{proof}
		We notice that for $A=B=0$  we have $f_{E,0,0,C,D}=f(X)$.
		We apply Theorem \ref{prod-polis}  choosing the tuples $$(a_1, b_1, c_1)=(1,0,0) , \,  (a_2, b_2, c_2)=(0,1,0), \text{ and } (a_3, b_3, c_3)=(0,0,1).$$ From Lemma \ref{Lemma_L}  the corresponding polynomials   $a_iX^{q^2} + b_iX^q + c_iX$, for $i=1,2,3 $, have no roots in $\mathbb{F}_{q^3}^*$.  This also yields $\alpha=\beta=\gamma=0$ and $\delta=1$ in Theorem \ref{prod-polis}. 
		
		Substituting these conditions into the system (\ref{sistema}), the problem reduces to verifying whether the coefficients of $f(X)$ satisfy the following equation
		\begin{equation}
			\label{eqX}
			4(C^3 + D^3 + E^3 - 3CDE) = 2,
		\end{equation}
		Since $q$ is an odd prime power, the equation \eqref{eqX} reduces to 
		\[
		C^3 + D^3 + E^3 - 3CDE = 2^{-1}. 
		\]
		Thus, under these conditions, $(0,0,C,D,E)$ is a valid solution to the system, and by Theorem \ref{prod-polis}, $f(X)$ is planar over $\mathbb{F}_{q^3}$.
	\end{proof}
	
	Now we present explicit examples of planar quadrinomials and pentanomials over $\mathbb{F}_{q^3}$.

	\begin{theorem}\label{teo1}
		Let $q$ be an odd prime power. Then, the quadrinomial
		\[
		f(X) =  -X^2 + 2 X^{q^2+1} + X^{2q} - X^{2q^2} \in \mathbb{F}_q[X]
		\]
		is planar over $\mathbb{F}_{q^3}$.
	\end{theorem}
	\begin{proof}
		We apply Theorem \ref{prod-polis}. Consider the tuples
		\[
		(a_1,b_1,c_1) = (1,-1,1), \quad (a_2,b_2,c_2) = (1,1,-1), \quad (a_3,b_3,c_3) = (-1,1,1).
		\]
		For each $i$, we have
		\[
		a_i^3 + b_i^3 + c_i^3 - 3a_ib_ic_i = 4 \neq 0,
		\]
		hence, by Lemma \ref{Lemma_L}, the corresponding polynomials $a_iX^{q^2} + b_iX^q + c_iX$ for $i=1,2,3 $ have no roots in $\F_{q^3}^*$. A direct computation gives
		\[
		\alpha = -1, \quad \beta = 1, \quad \gamma = 1, \quad \mbox{and} \quad \delta = -2.
		\]
		in Theorem \ref{prod-polis}. Choosing $B \in \F_q^*$ satisfying $B^3 = 4$, and the 5-tuple given by
		\[
		(A,B,C,D,E) = \left(0,\; B,\; \frac{2}{B^2},\; -\frac{B}{2},\; -\frac{B}{2}\right).
		\]
		it can be verified that it satisfies the system \eqref{sistema}. In fact, since $B^3=4$, 
		substituting into the system \eqref{sistema}, we obtain
		
		$$	\begin{cases} 
			&-B^2C = -2 = 2\alpha,\\
			&B^2E - 2BCE + 2BD^2 = 2 = 2\beta,\\
			&B^2D - 2BCD + 2BE^2 = 2 = 2\gamma,\\
			&B^3 + 4C^3 - 12CDE + 4D^3 + 4E^3 = -4 = 2\delta.
		\end{cases} 
		$$
		
		Therefore, by Theorem \ref{prod-polis}, the polynomial
		\[
		f_{-\frac{B}{2},\; 0,\; B,\; \frac{2}{B^2},\; -\frac{B}{2} }(X) = - \frac{B}{2}X^2 + B X^{q^2+1} + \frac{2}{B^2}X^{2q} - \frac{B}{2}X^{2q^2}
		\]
		is planar over $\F_{q^3}$. So, also $f(X)=B^2/2f_{-\frac{B}{2},\; 0,\; B,\; \frac{2}{B^2},\; -\frac{B}{2}}(X)$ is planar.
	\end{proof}

	We now present a family of planar pentanomials with two parameters.
	
	\begin{theorem}\label{2parametros}
	
	Let $q$ be an odd prime power, and let $D, E$ in  $\mathbb{F}_q$. Define $\omega = 3D^2 - 3DE + E^2$ and consider the polynomial
	$$f(X) =  EX^2 + 2(E-D)X^{q+1} + 2DX^{q^2+1} + (E-D)X^{2q} + DX^{2q^2}.
	$$
	Then, $f(X)$ is planar over $\mathbb{F}_{q^3}$ if and only if $E\omega \neq 0$.
	\end{theorem}
	\begin{proof}	
	From Lemma \ref{Lemma_L}, the polynomials
	$$4X^{q^2} + 4X^q, \quad EX^{q^2} + EX, \quad \omega X^q + \omega X$$
	have no roots in $\mathbb{F}_{q^3}^*$. These polynomials satisfy the hypothesis of Theorem \ref{prod-polis}, and we obtain
	
		$$\alpha = 0, \quad \beta = \gamma = 4E\omega, \quad \delta = 8E\omega.$$
	A straightforward analysis shows that the coefficients of $f_{E,2(E-D), 2D, E-D, D}(X)$ satisfy the corresponding system of equations, and so $f_{E,2(E-D), 2D, E-D, D}(X)$ is planar over $\mathbb{F}_{q^3}$.

	\end{proof}	
	
		In the last theorem we provide two new pentanomials planar functions.  
	
		\begin{theorem} \label{teo: two pent}
		Let $q$ be a power of an odd prime.  Then the polynomials
		\begin{enumerate}
			\item $f_{1,-1,-1,1,1}(X) = X^2 - X^{q+1} - X^{q^2+1} + X^{2q} + X^{2q^2}$ and 
			\item $f_{1,1,1,1/2,1/2}(X) = X^2 + X^{q+1} + X^{q^2 + 1} + \frac{1}{2}X^{2q} + \frac{1}{2} X^{2q^2}$
		\end{enumerate}
			are planar over $\mathbb{F}_{q^3}$.
	\end{theorem}	
	\begin{proof}
		We apply Proposition \ref{Prop:Condition}. For the function $f_{1,-1,-1,1,1}(X)$ and
		 $\e \in \mathbb{F}_{q^3}^{\ast}$ we have 
		$
		\Delta_{F_\e}= B_\e \neq 0$ following the notation and results of Proposition  \ref{lemma Bh}.	Similarly, for the polynomial $f_{1,1,1,1/2,1/2}(X)$ and $\epsilon \in \mathbb{F}_{q^3}^{\ast}$, we have $\Delta_{F_\epsilon} = \det(A_\epsilon) \neq 0$, following the notation and results of Proposition \ref{lemma Ag}.
	\end{proof}	
	
	\section*{Acknowledgements}
	The first author was partially supported by FAPESP (Brazil), grant no. 2024/19443-4. The third author was partially supported by Conselho
Nacional de Desenvolvimento Científico e Tecnológico  CNPq/Brasil - 307261/2023-9, Fundação de Amparo a Pesquisa do Estado do Rio de Janeiro FAPERJ/Brasil (CNE E-26/204.037/2024), and Coordenação de Aperfeiçoamento de Pessoal de Nível Superior CAPES/Brasil, 001. The fourth author was partially supported by CNPq Grant 306270/2023-4 and FAPEMIG Grant APQ-01443-23.

	\end{document}